\documentclass{amsart}
\usepackage{lmodern}
\usepackage[english]{babel}
\usepackage{amsmath}
\usepackage{amssymb}
\usepackage{mathrsfs}
\usepackage{color}
\usepackage{graphicx}
\usepackage{caption}
\usepackage{subcaption}
\usepackage{float}
\usepackage[all,cmtip]{xy}
\usepackage{bm}
\usepackage{algorithm}
\usepackage{enumitem}
\newlist{steps}{enumerate}{1}
\setlist[steps, 1]{label = Step \arabic*:}
\usepackage{pinlabel}
\newtheorem{theorem}{\bf Theorem}[section]
\newtheorem{lemma}[theorem]{\bf Lemma}
\newtheorem{definition}[theorem]{\bf Definition}
\newtheorem{corollary}[theorem]{\bf Corollary}
\newtheorem{proposition}[theorem]{\bf Proposition}
\newtheorem{remark}[theorem]{\bf Remark}

\newtheorem{question}[theorem]{\bf Question}

\newcommand{\rme}{\mathrm{e}}
\newcommand{\rmi}{\mathrm{i}}

\begin{document}

\title{Ultraknots and limit knots}

\author{
Benjamin Bode}
\date{}

\address{Instituto de Ciencias Matemáticas (ICMAT), Consejo Superior de Investigaciones Científicas (CSIC), C/ Nicolás Cabrera, 13-15, Campus Cantoblanco, UAM, 28049 Madrid, Spain}
\email{benjamin.bode@icmat.es}








%

\maketitle
\begin{abstract}
We prove that every knot type in $\mathbb{R}^3$ can be parametrised by a smooth function $f:S^1\to\mathbb{R}^3$, $f(t)=(x(t),y(t),z(t))$ such that all derivatives $f^{(n)}(t)=(x^{(n)}(t),y^{(n)}(t),z^{(n)}(t))$, $n\in\mathbb{N}$, parametrise knots and every knot type appears in the corresponding sequence of knots. We also study knot types that arise as limits of such sequences.
\end{abstract}

 \vspace{0.2cm} 

   
     \vspace{0.2cm} 
     

\section{Introduction}

A knot $K$ in $\mathbb{R}^3$ is the image of a smooth map $f:S^1\cong\mathbb{R}/2\pi\to\mathbb{R}^3$, $f(t)=(x(t),y(t),z(t))$ with $(x(t),y(t),z(t))\neq (x(s),y(s),z(s))$ for all $s\neq t$. Taking derivatives with respect to $t$, we obtain another space curve $K_1$, parametrised by $(x'(t),y'(t),z'(t))$. This loop $K_1$ is not necessarily a knot, since there could be self-intersections or parts where the loop traces back on itself. Still, we may continue to take derivatives with respect to $t$ and for every $n\in\mathbb{N}$ the $n$th derivative $f^{(n)}(t)=(x^{(n)}(t),y^{(n)}(t),z^{(n)}(t))$ parametrises a loop $K_n$ in $\mathbb{R}^3$, so that if these loops are simple, we obtain a sequence of knots. This paper studies the question of which sequences of knots can arise in this way.

Throughout the paper all knots are assumed to be in $\mathbb{R}^3$ and smooth. If two knots $K$ and $L$ are ambient isotopic, this is written as $K\cong L$.

\begin{definition}
Let $f:S^1\to\mathbb{R}^3$, $f(t)=(x(t),y(t),z(t))$ be a smooth function, parametrising a knot $K$. Let $K_n$ be the loop in $\mathbb{R}^3$ that is parametrised by $f^{(n)}(t)=(x^{(n)}(t),y^{(n)}(t),z^{(n)}(t))$. We call $K$ an \textbf{ultraknot} if $K_n$ is a knot for all $n\in\mathbb{N}$ and for every knot $L$ there is an $n\in\mathbb{N}$ so that $K_n$ is ambient isotopic to $L$. 
\end{definition}

The definition, the terminology and the following question are all due to Peter Feller:
\begin{question}
Which knots are ultraknots?
\end{question}

We answer this question and show that not only is every knot an ultraknot, but we can also prescribe the order in which different knot types appear in the sequence.

\begin{theorem}\label{thm:ultra}
Let $K$ be a knot and $(L_i)_{i\in\mathbb{N}}$ be any sequence of knots. Then there is a smooth map $f:S^1\to\mathbb{R}^3$ that parametrises $K$, the derivatives $f^{(n)}$ parametrise knots $K_n$ for all $n$ and the sequence $L_i$ is a subsequence of $(K_n)_{n\in\mathbb{N}}$, that is, there is a strictly monotone increasing sequence $(n_i)_{i\in\mathbb{N}}$ of natural numbers such that $K_{n_i}$ is ambient isotopic to $L_{i}$ for all $i\in\mathbb{N}$.
\end{theorem}
Since for every natural number $n$ there are finitely many knots whose minimal crossing number is $n$, the set of knot types is countable. We can thus pick a sequence of knots $L_i$ that contains every knot type.
\begin{corollary}
Every knot is ambient isotopic to an ultraknot.
\end{corollary}

The conceptual opposite of an ultraknot, whose sequence of knots contains all knot types, is a knot whose sequence of knots consists of only one single knot type.

\begin{definition}
Let $f:S^1\to\mathbb{R}^3$, $f(t)=(x(t),y(t),z(t))$ be a smooth function, parametrising a knot $K$ in $\mathbb{R}^3$. We call the sequence of loops $(K_n)_{n\in\mathbb{N}}$ in $\mathbb{R}^3$ that are parametrised by $f^{(n)}(t)=(x^{(n)}(t),y^{(n)}(t),z^{(n)}(t))$ the \textbf{loop sequence} of $f$. 
A knot L is called a \textbf{limit knot} if there is a smooth function $f:S^1\to\mathbb{R}^3$ whose loop sequence converges to $L$, that is, $K_n$ is ambient isotopic to $L$ for all sufficiently large $n$.
\end{definition}

The terminology is inspired by the analogous setting of convergent sequences of points. Using the discrete topology on the set of all knot types, a sequence of knot types converges if and only if it becomes eventually constant and the set of possible knots that a sequence can stabilise to is exactly the set of limit knots.

In particular, a knot is a limit knot if and only if it can be parametrised such that the resulting sequence of knot types is constant.
\begin{question}
Which knots are limit knots?
\end{question}

At this moment we do not have a complete answer to this question. However, when we restrict to maps $f:S^1\to\mathbb{R}^3$ whose coordinate functions $(x(t),y(t),z(t))$ are given by trigonometric polynomials, or equivalently by finite Fourier series, we obtain an interesting connection to the study of Lissajous knots, introduced in \cite{272}.

A knot is called a \textbf{Lissajous knot} if it can be parametrised as
\begin{align}
x(t)&=\cos(n_xt+\varphi_x),\nonumber\\
y(t)&=\cos(n_yt+\varphi_y),\\
z(t)&=\cos(n_zt+\varphi_z)\nonumber
\end{align}
for some $n_x,n_y,n_z\in\mathbb{N}\cup\{0\}$ and $\varphi_x,\varphi_y,\varphi_z\in[0,2\pi)$. Without loss of generality we can take $\varphi_x=0$.

Suppose that $f:S^1\to\mathbb{R}^3$ has coordinate functions $(x(t),y(t),z(t))$ that are trigonometric polynomials whose highest order terms are given by $A_{x,n_x}\cos(n_xt)$, $A_{y,n_y}\cos(n_yt+\varphi_y)$ and $A_{z,n_z}\cos(n_zt+\varphi_z)$ with $A_{x,n_x}, A_{y,n_y},A_{z,n_z}\in\mathbb{R}\backslash\{0\}$, $n_x,n_y,n_z\in\mathbb{N}$ and $\varphi_y,\varphi_z\in[0,2\pi)$. Assume that the frequencies $n_x$, $n_y$ and $n_z$ are pairwise coprime, i.e., no pair of them has a common prime factor. Furthermore, assume that $\varphi_y/\pi$, $\varphi_z/\pi$ and $(\varphi_y-\varphi_z)/\pi$ are all irrational. We call a knot a \textbf{Fourier limit knot} if it is the limit knot of the loop sequence of such a map $f$. In particular, all Fourier limit knots are limit knots.



\begin{theorem}\label{thm:limit}
Let $K$ be a Fourier limit knot. Then $K$ is a Lissajous knot. Conversely, if $K$ is a Lissajous knot such that all frequencies $n_x$, $n_y$ and $n_z$ are odd, then $K$ is a Fourier limit knot.
\end{theorem}

The necessary condition of being a Lissajous knot imposes certain restrictions on Fourier limit knots. It follows immediately that certain knots like the trefoil or the figure-eight knot are not Fourier limit knots \cite{272}. However, it is not known if they are limit knots of loop sequences of functions that are not trigonometric polynomials.

The rest of the paper is structured as follows. In Section~\ref{sec:2} we prove that given any knot type $K$ and sequence of links $L_i$ we can parametrise $K$ such that the corresponding loop sequence contains $L_i$ as a subsequence. This procedure is completely algorithmic. However, for some derivatives the corresponding loop $K_n$, $n\neq n_i$ for all $i$, might not be a knot. In Section~\ref{sec:3} we prove that after a small deformation of the function from Section~\ref{sec:2} all $K_n$ are knots and the $L_i$s can still be found as a subsequence. This proves Theorem~\ref{thm:ultra}. In Section~\ref{sec:4} we study Fourier limit knots and prove Theorem~\ref{thm:limit}. We also prove that every knot has a parametrisation whose loop sequence converges to the unknot, see Corollary~\ref{cor:unknot}.

\section{Subsequences of knots}\label{sec:2}

In this section we prove the following Proposition.
\begin{proposition}\label{prop}
Let $K$ be a knot and $(L_i)_{i\in\mathbb{N}}$ be a sequence of knots. Then there is a smooth map $f:S^1\to\mathbb{R}^3$ that parametrises $K$ and there is a strictly monotone increasing sequence $(n_i)_{i\in\mathbb{N}}$ of natural numbers such that $K_{n_i}$, parametrised by $f^{(n_i)}$, is ambient isotopic to $L_i$ for all $i\in\mathbb{N}$.
\end{proposition}
Note that the proposition does not automatically imply that $K$ is an ultraknot because $K_n$, given by $f^{(n)}$, could be a non-simple loop when $n\neq n_i$ for all $i$.

A trigonometric polynomial $x:S^1\to\mathbb{R}$ can be written as a finite linear combination of $1$, $\cos(kt)$, $\sin(kt)$, $k\in\mathbb{N}$, over the reals:
\begin{equation}
x(t)=a_0+\sum_{k=1}^{N}(a_k\cos(kt)+b_k\sin(kt)).
\end{equation} 
We call $N$ the degree of the trigonometric polynomial $x$ and the lowest $k$ with $a_{k}\neq 0$ and $b_k\neq 0$ the lowest order of $x$.

\begin{lemma}\label{lem:param}
Let $K$ be a knot and $n\in\mathbb{N}$. Then there are trigonometric polynomials $x(t),y(t),z(t)$ all of which have a lowest order at least $n$ such that $(x(t),y(t),z(t))$ parametrises $K$.
\end{lemma}
\begin{proof}
In \cite{bode:algo} we present an algorithm that constructs for any braid $B$ that closes to a knot $K$ a pair of trigonometric polynomials $F,G:S^1\to\mathbb{R}$ that can be used to parametrise $B$. If $B$ is a braid on $n$ strands, then
\begin{align}
x(t)&=\cos(nt)(R+F(t)),\nonumber\\
y(t)&=\sin(nt)(R+F(t)),\\
z(t)&=G(t)\nonumber
\end{align}
is a parametrisation of $K$, where $R$ is a positive real number that is bigger than $\max_{t\in S^1}|F(t)|$. We see that $x$, $y$ and $z$ are all trigonometric polynomials and the lowest order of $x$ and $y$ is at least $n$. We need to replace $G$ with a different trigonometric polynomial whose lowest order is also bigger than $n$.

As in \cite{bode:ak} the trigonometric polynomial $F$ can be constructed such that $(x(t),y(t))$ parametrises a regular knot diagram in $\mathbb{R}^2$, i.e., there are only finitely many crossings and each of them is a transverse intersection of exactly two strands. We thus need to find a trigonometric polynomial $\widetilde{G}$ of lowest order bigger than $n$ that results in the desired crossing signs, i.e., the same crossing signs as $G$.

A crossing in the diagram parametrised by $(x(t),y(t))$ corresponds to values $0\leq t<s<2\pi$ with $(x(t),y(t))=(x(s),y(s))$ and the sign of the desired crossing is determined by the sign of $G(t)-G(s)$. We write $t_j, s_j\in\mathbb{R}$, $j=1,2,\ldots,m$, for the set of values of $t$ corresponding to crossings, i.e., $t_j<s_j$ and $(x(t_j),y(t_j))=(x(s_j),y(s_j))$. Note that for every finite set of points $(t_j)_{j=1,2,\ldots,m}\cup(s_j)_{j=1,2,\ldots,m}$ in $[0,2\pi)$ and every $n\in\mathbb{N}$ there is an $N>n$ and a phase shift $\varphi\in[0,2\pi)$ so that $\cos(Nt_j+\varphi)$ is positive for all $j$ and $\cos(Ns_j+\varphi)$ is positive for all $j$. We can write $\cos(Nt+\varphi)$ as a linear combination of $\cos(Nt)$ and $\sin(Nt)$, say $\cos(Nt+\varphi)=A\cos(Nt)+B\sin(Nt)$ with $A,B\in\mathbb{R}$. But then $\widetilde{G}(t):=(A\cos(Nt)+B\sin(Nt))G(t)$ is a trigonometric polynomial of lowest order at least $N>n$ and such that the sign of $\widetilde{G}(t_j)-\widetilde{G}(s_j)$ is the same as the sign of $G(t_j)-G(s_j)$ for all $j$.

It follows that
\begin{align}
x(t)&=\cos(nt)(R+F(t)),\nonumber\\
y(t)&=\sin(nt)(R+F(t)),\\
z(t)&=\widetilde{G}(t)\nonumber
\end{align}
is the desired parametrisation of $K$.
\end{proof}

Knots that are parametrised by trigonometric polynomials are called Fourier knots. Rewriting cosines and sines as linear combinations of $\rme^{\rmi kt}$ and $\rme^{-\rmi kt}$, we can represent every trigonometric polynomial as a finite Fourier series $\sum_{k=-N}^N c_k\rme^{\rmi kt}$, where the complex coefficients $c_k$ are related to $a_k$ and $b_k$ in the expression above via $c_0=a_0$, $c_k=(a_k-\rmi b_k)/2$ if $k>0$ and $c_k=(a_{-k}+\rmi b_{-k})/2$ if $k<0$. In particular, $\text{Im}(c_k)=-\text{Im}(c_{-k})$ and $\text{Re}(c_k)=\text{Re}(c_{-k})$ for all $k$.

As a generalisation of Fourier knots, we may consider knots that are parametrised by smooth (in particular, convergent) infinite Fourier series $\sum_{k=-\infty}^{\infty}c_k\rme^{\rmi kt}$.

\begin{lemma}\label{lem:epsilon}
Let $K$ be a knot that is parametrised by a triple of smooth Fourier series $f(t)=(x(t),y(t),z(t))$. Then there exists an $\varepsilon>0$ such that for every triple of smooth Fourier series $g(t)=(x_1(t),x_2(t),x_3(t))$, $x_i=\sum_{k=-\infty}^{\infty}c_{i,k}\rme^{\rmi kt}$, with $|c_{i,k}|<\tfrac{\varepsilon}{2^{|k|}}$ for all $k$ and all $i\in\{1,2,3\}$ the triple of smooth Fourier series $f(t)+g(t)$ is a smooth parametrisation of $K$.
\end{lemma}
\begin{proof}
The space of $C^1$-functions $f:S^1\to\mathbb{R}^3$ is endowed with the $C^1$-norm: 
\begin{equation}
|f|_1:=\max_{t\in S^1}\max\{|f(t)|,|f'(t)|\}.
\end{equation} 
It is a known fact that the subset of functions $f$ that parametrise a given knot type is an open set with respect to the topology induced by this norm. Therefore for every knot parametrised by a triple of smooth Fourier series there is an $\varepsilon'$ such that for all triples of smooth Fourier series $g(t)$ with $|g|_1<\varepsilon'$ we obtain a new parametrisation of $K$ from $f(t)+g(t)$.

Let now $\varepsilon=\varepsilon'/15$ and assume that $g(t)$ is a triple of smooth Fourier series $(x_1(t),x_2(t),x_3(t))$, $x_i=\sum_{k=-\infty}^{\infty}c_{i,k}\rme^{\rmi kt}$, with $|c_{i,k}|<\tfrac{\varepsilon}{2^{|k|}}$ for all $k$ and all $i\in\{1,2,3\}$. We claim that $|g|_1<\varepsilon'$, which then proves the lemma.

We have
\begin{align}
|g|_1\leq \sum_{i=1}^3|x_i|_1&=\sum_{i=1}^3 \max\{\max_{t\in[0,2\pi]}\{|x_i(t)|,|x_i'(t)|\}\}\nonumber\\
&\leq\sum_{i=1}^3 \max\{\sum_{k=-\infty}^{\infty}|c_{i,k}|,\sum_{k=-\infty}^{\infty}|kc_{i,k}|\}\nonumber\\
&\leq \sum_{i=1}^3 \sum_{k=-\infty}^{\infty}\max\{|c_{i,k}|,|kc_{i,k}|\}\nonumber\\
&\leq \sum_{i=1}^3 (|c_{i,0}|+\underset{k\neq 0}{\sum_{k=-\infty}^{\infty}}|kc_k|)\nonumber\\
&< 3\left(\varepsilon+ \underset{k\neq 0}{\sum_{k=-\infty}^{\infty}}\frac{\varepsilon k}{2^k}\right)=3\varepsilon (1+4), 
\end{align}

where the last equality follows from $\sum_{k=1}^{\infty}\tfrac{|k|}{2^k}=2$. By definition we have $15\varepsilon=\varepsilon'$. Therefore, $|g|_1<\varepsilon'$ and $f(t)+g(t)$ is a parametrisation of $K$.
\end{proof}

\begin{remark}\label{remark}
We may multiply any knot parametrisation $(x(t),y(t),z(t))$ by some positive real number $A$ and still obtain a parametrisation of the same knot. The $\varepsilon$-value from Lemma~\ref{lem:epsilon} of the new curve is $A$ times the $\varepsilon$-value of the old curve $(x(t),y(t),z(t))$. In particular, for a given knot type we may assume that $\varepsilon$ is equal to 1. If the coordinate functions are trigonometric polynomials, then multiplication by $A$ does not change their lowest order or degrees. It follows that Lemma~\ref{lem:param} is true even if we demand that the $\varepsilon$-value of $K$ from Lemma~\ref{lem:epsilon} is equal to 1.
\end{remark}

Before we prove Proposition~\ref{prop} we briefly outline the main idea of the proof. Suppose $\sum_{k=-\infty}^{\infty}c_k\rme^{\rmi kt}$ is a smooth Fourier series whose coefficients $c_k$ converge rapidly to 0 as $|k|$ increases. In particular, if there are three such Fourier series that together parametrise a knot $K$, then the lower order terms of the Fourier series are dominant. That is to say, we could ignore all terms above a certain degree and obtain a parametrisation of the same knot $K$. The $n$th derivative then has Fourier coefficients $(\rmi k)^nc_k$ and so the absolute value of the coefficients grows as $n$ increases. However, coefficients of low orders (i.e., low absolute value of $k$) grow much slower than those with large $|k|$. Thus there should be a range of values of $n$ where the lower order terms are negligible (because $|k|$ is small) and very high order terms are also negligible (because $c_k$ is very small). Thus the dominant terms lie in some middle range and as we increase $n$, the range of values of $k$ that are relevant for the topology of the knot that is parametrised by $f^{(n)}$ shifts to higher values of $|k|$. We should therefore combine a trigonometric parametrisation of a knot $K$ with trigonometric parametrisations of the different knots $L_i$ in such a way that the degree of the parametrisation of $K$ is less than the lowest order of $L_1$ and the degree of the parametrisation of $L_i$ is less than the lowest order of $L_{i+1}$ for all $i$. Of course, this rough sketch ignores a lot of technical issues, such as, what it means for a coefficient to be small or questions about convergence. 

\begin{proof}[Proof of Proposition~\ref{prop}]
Let $K$ be a knot and let $(L_i)_{i\in\mathbb{N}}$ be a sequence of knots. Using the procedure from \cite{bode:algo} we can find a parametrisation 
\begin{equation}
f_0(t):=(x_0(t),y_0(t),z_0(t))
\end{equation} 
of $K$ in terms of trigonometric polynomials. Let $M_0$ be the maximal degree of these three trigonometric polynomials. By Lemma~\ref{lem:param} we can find a parametrisation $f_1(t)=(x_1(t),y_1(t),z_1(t))$ of $L_1$ in terms of trigonometric polynomials whose lowest order $N_1$ is strictly greater than $M+1$. We define $M_1$ to be the maximal degree of $x_1$, $y_1$ and $z_1$. Inductively, we can find a trigonometric parametrisation $f_i(t)=(x_i(t),y_i(t),z_i(t))$ of $L_i$ whose lowest order $N_i$ is strictly greater than $M_{i-1}+1$, where $M_{i-1}$ is the maximal degree of $x_{i-1}$, $y_{i-1}$ and $z_{i-1}$.

Let $(n_i)_{i\in\mathbb{N}\cup\{0\}}$ be a strictly monotone increasing sequence of natural numbers. If $x_i(t)=\sum_{k=-M_i}^{M_i}c_{i,k}\rme^{\rmi kt}$, we set 
\begin{equation}
\widetilde{x_i}(t)=\sum_{k=-M_i}^{M_i}\left(\tfrac{(N_i-1)}{\rmi k}\right)^{n_i}c_{i,k}\rme^{\rmi kt}
\end{equation} 
and define $\widetilde{y_i}(t)$ and $\widetilde{z_i}(t)$ analogously. Note that $c_{i,k}=0$ if $|k|<N_i$. We write $\widetilde{f_i}(t)=(\widetilde{x_i}(t),\widetilde{y_i}(t),\widetilde{z_i}(t))$. Note that in particular, we may set $n_0=0$, which yields $\widetilde{f_0}=f_0$.

We claim that there is a strictly monotone increasing sequence of natural numbers $(n_i)_{i\in\mathbb{N}}$ such that $f(t):=\sum_{i=0}^{\infty} \widetilde{f_i}(t)$ is a parametrisation of $K$ and $F^{(n_i)}(t)$ is a parametrisation of $L_i$ for all $i\in\mathbb{N}$.

By Lemma~\ref{lem:epsilon} there exists an $\varepsilon_0>0$ such that adding Fourier series with coefficients that are smaller than an expression in $\varepsilon_0$ to $f$ does not change the knot $K$. 

Likewise, let $\varepsilon_i$ be the positive real number from Lemma~\ref{lem:epsilon} for the knot $L_i$ parametrised by $f_i(t)$. Note that $\widetilde{f_i}^{(n_i)}(t)=(N_i-1)^{n_i}f_i(t)$ (for all $i>0$) and thus it is also a parametrisation of $L_i$ with corresponding $\varepsilon$-value $(N_i-1)^{n_i}\varepsilon_i$.

Therefore, we need to choose the $n_i$s such that 
\begin{equation}\label{eq:ninj}
(N_i-1)^{n_i}|k|^{n_j-n_i}|c_{i,k}|<(N_j-1)^{n_j}\varepsilon_j/2^{|k|}
\end{equation} 
for all $i,j\in\mathbb{N}\cup\{0\}$ with $i\neq j$, where the left hand side is the coefficient of $\rme^{\rmi kt}$ in $f^{(n_i)}(t)$. Note that this equation needs to hold both for $i<j$ and $j>i$. The first guarantees that $f^{(n_j)}$ parametrises $L_j$, while the latter guarantees that the terms in $f$ that correspond to $L_j$ do not change the fact that $f^{(n_i)}$ parametrises $L_i$.

We may choose the values of the $n_i$s successively, starting with $n_1$ so that Eq.~\eqref{eq:ninj} is satisfied for $i=0$, $j=1$, as well as for $i=1$, $j=0$, where $n_0=0$. The case of $i=1$, $j=0$ is satisfied if $\left|\left(\tfrac{(N_1-1)}{k}\right)^{n_1}c_{1,k}\right|<\tfrac{\varepsilon_0}{2^{|k|}}$ holds for all $k$. Since $c_{1,k}=0$ unless $|k|>N_1-1$, this equation is satisfied as long as $n_1$ is sufficiently large, namely 
\begin{equation}
n_1>\log_{\tfrac{|k|}{(N_1-1)}}\left(\frac{2^{|k|}|c_{1,k}|}{\varepsilon_0}\right)
\end{equation}
for all $k$ with $c_{1,k}\neq 0$. 
The case of $i=0$, $j=1$ is satisfied if
\begin{equation}
k^{n_1}|c_{0,k}|<(N_1-1)^{n_1}\varepsilon_1/2^{|k|}
\end{equation}
for all $k$, which is equivalent to $\left(\tfrac{(N_1-1)}{|k|}\right)^{n_1}>|c_{0,k}|2^{|k|}/\varepsilon_0$ for all $k$. Since $N_1-1>M_0\geq |k|$ for all $k$ with $c_{0,k}\neq 0$, we can achieve this by choosing $n_1$ large enough. 

Suppose now that we have picked the first $p$ values of the sequence $(n_i)_{i\in\mathbb{N}}$ so that Eq.~\eqref{eq:ninj} holds for all $i\neq j$ with $i,j\leq p$. Then we pick $n_{p+1}$ so that 
\begin{equation}
(N_{p+1}-1)^{n_{p+1}}k^{n_j-n_{p+1}}|c_{p+1,k}|<(N_j-1)^{n_j}\tfrac{\varepsilon_j}{2^{|k|}}
\end{equation} 
and
\begin{equation}
(N_{j}-1)^{n_j}k^{n_{p+1}-n_j}|c_{j,k}|<(N_{p+1}-1)^{n_{p+1}}\varepsilon_{p+1}/2^{|k|}
\end{equation}
for all $k$ and $j<p+1$. Again it is sufficient to pick $n_{p+1}$ large enough, since $c_{p+1,k}=0$ unless $k>N_{p+1}-1$ and $c_{j,k}\neq 0$ implies $k<M_j\leq N_{p+1}-1$ for all $j<p+1$.

Thus for every $i\in\mathbb{N}$ the function $f^{(n_i)}$ is a close approximation to $\widetilde{f_i}^{(n_i)}=(N_i-1)^{n_i}f_i$ and thus a parametrisation of $L_i$.
\end{proof}

\begin{remark}\label{rem2}
It is sufficient for the proof of Proposition~\ref{prop} to choose the sequence $(n_i)_{i\in\mathbb{N}}$ such that Eq.~\eqref{eq:ninj} is always satisfied. We prove that this is possible for any sequence of positive real numbers $\varepsilon_j$. As mentioned in Remark~\ref{remark} we may set $\varepsilon_j=1$ for all $j$. Note that we can still choose the $n_i$s such that
\begin{equation}
(N_i-1)^{n_i}|k|^{n_j-n_i}|c_{i,k}|<(N_j-1)^{n_j}1/2^{|k|+1}
\end{equation}
holds for all $i,j\in\mathbb{N}\cup \{0\}$ with $i\neq j$. Thus we pick the $n_i$s such that the resulting approximations are twice as close to the $L_i$s as would be strictly necessary. This leaves us some extra wiggle room for the additional modifications that we are about to do in the next section.
\end{remark}

\section{Avoiding self-intersections}\label{sec:3}

In the previous section we showed that for every knot $K$ and every sequence of knots $(L_i)_{i\in\mathbb{N}}$ there exists a parametrisation $f:S^1\to\mathbb{R}^3$ of $K$ whose sequence of loops $(K_n)_{n\in\mathbb{N}}$ contains $(L_i)_{i\in\mathbb{N}}$ as a subsequence. This means that there is a strictly monotone increasing sequence $(n_i)_{i\in\mathbb{N}}$ of natural number such that $K_{n_i}\cong L_i$ for all $i\in\mathbb{N}$. In order to show that $K$ is an ultraknot, we have to fill the gaps between the $n_i$s, that is, we have to make sure that $K_n$ is a knot, not some non-simple loop, even when $n\neq n_i$ for all $i\in\mathbb{N}$.

We want to emphasise that we require that $f^{(n)}$ is injective. We are not satisfied with the image of $f^{(n)}$ being a knot. These two notions are not equivalent, since $f^{(n)}$ could parametrise a loop that traverses its image several times.

\begin{definition}
We say that a parametric curve $f:S^1\to\mathbb{R}^n$ is 1-covered if there is a point $y\in f(S^1)$ such that $f^{-1}(y)$ is a unique point on $S^1$. We say that $f$ is $N$-covered for some $N\in\mathbb{N}_{>1}$ if $f$ is the composition of an $N$-fold covering map of $S^1$ and a 1-covered parametric curve.
\end{definition}

We now study the self-intersection points of a real-analytic planar parametric curve $(x(t),y(t))$. Such intersection points correspond to values $t_*\neq s_*$ with $(x(t_*),y(t_*))=(x(s_*),y(s_*))$. Consider the curve $(\cos(nt),\sin(nt))$. It has infinitely many self-intersection points, but only because it is $n$-covered. Its image is a smooth manifold without any self-intersections, meaning there are no $t_*\neq s_*$ in the fundamental domain $[0,2\pi/n)$ with $(x(t_*),y(t_*))=(x(s_*),y(s_*))$.

\begin{lemma}\label{Nself}
Let $f(t)=(x(t),y(t)):S^1\to\mathbb{R}^2$ be a real-analytic map. Assume that $f$ is regular, i.e., there are no values of $t$ with $f'(t)=(0,0)$. Then the image of $f$ is a loop with at most finitely many self-intersections.
\end{lemma}
\begin{proof}
Since $f$ is $2\pi$-periodic, its image is a loop. Now assume that there are infinitely many self-intersections. Then there are sequences of values $(t_n)_{n\in\mathbb{N}}$ and $(s_n)_{n\in\mathbb{N}}$, converging to $t_*\in[0,2\pi]$ and $s_*\in[0,2\pi]$, respectively, and such that $f(t_n)=f(s_n)$ for all $n$.

Since $f$ is regular, there is a neighbourhood $U$ of $t_*$ such that $f$ is injective on $U$. In particular, $U$ does not contain any $s_n$ with $n$ such that $t_n\in U$. Therefore $s_*\notin U$ and in particular $s_*\neq t_*$. Thus there is an $\varepsilon>0$ such that $[t_*-\varepsilon,t_*+\varepsilon]\subset\mathbb{R}$ and $[s_*-\varepsilon,s_*+\varepsilon]\subset\mathbb{R}$ are disjoint and $f$ is invertible on each of the two intervals. Note that we consider these intervals as subsets of $\mathbb{R}$ (as opposed to $[0,2\pi]$) to deal for example with the case where $t_*=0$.

Since $f(t)$ is real-analytic, it has a complex analytic extension on some open neighbourhood $V_1$ of $[t_*-\varepsilon,t_*+\varepsilon]$ in $\mathbb{C}$. Similarly, is has a complex analytic extension on some open neighbourhood of $V_2$ of $[s_*-\varepsilon,s_*+\varepsilon]$ in $\mathbb{C}$. By continuity, the extension of $f$ remains regular, that is, invertible, on these neighbourhoods. Since $f$ is analytic, the local inverse is also analytic.

We write $f_t$ for the restriction of $f$ to $V_1$ and $f_s$ for the restriction of $f$ to $V_2$. We define $W:=f(V_1)\cap f(V_2)$. Consider now the map $f_t^{-1}\circ f_s:V_2\to V_1$. It is an analytic map such that intersection points between $f([t_*-\varepsilon,t_*+\varepsilon])$ and $f([s_*-\varepsilon,s_*+\varepsilon])$ are precisely $[t_*-\varepsilon,t_*+\varepsilon]\cap f_t^{-1}(f_s([s_*-\varepsilon,s_*+\varepsilon]))$. In particular, we have $\text{Im}(f_t^{-1}\circ f_s)(s_n)=0$ for all $s_n\in V_2$ and thus $\text{Im}(f_t^{-1}\circ f_s)(s_*)=0$. Since $\text{Im}(f_t^{-1}\circ f_s)$ is real analytic on $[s_*-\varepsilon,s_*+\varepsilon]$, this implies that $\text{Im}(f_t^{-1}\circ f_s)$ is constant 0 on $[s_*-\varepsilon,s_*+\varepsilon]$.

Now consider $\hat{f}(t):S^1\cong\mathbb{R}/L\to\mathbb{R}^2$, the arc-length-parametrisation of $f(t)$, so that $|\hat{x}'(t)|^2+|\hat{y}'(t)|^2=1$ for all $t$, where $\hat{f}(t)=(\hat{x}(t),\hat{y}(t))$. Here $L$ is the length of the curve. This is well-defined, since $f$ is non-regular. The map $\hat{f}$ is unique and real-analytic. Let $\hat{t}_*$ and $\hat{s}_*$ be the images of $t_*$ and $s_*$ under the reparametrisation map $[0,2\pi]\to[0,L]$. We then have that $\hat{f}(t)$ and $\hat{f}(t-\hat{s}_*+\hat{t}_*)$ are both real-analytic and agree on the image of $[s_*-\varepsilon,s_*+\varepsilon]$ under the reparametrisation map. But then they must agree on their entire domain $[0,L]$.

Thus the map $\hat{f}$ is $(\hat{t}_*-\hat{s}_*)$-periodic. It is therefore the composition of the map $t\mapsto k t$, for some natural number $k>1$ and a map $g:S^1\to\mathbb{R}^2$ that has at least one point in its image with unique preimage point. Applying the same arguments as above to $g$, shows by contradiction that the image of $g$ has only finitely many self-intersection points. By definition the image of $g$ is equal to the image of $f$, which proves the lemma.
\end{proof}

\begin{lemma}\label{curvature}
Let $f(t):S^1\to\mathbb{R}^2$, $f(t)=(x(t),y(t))$ be a real-analytic, non-constant map. Then there is a natural number $N$ such that for every $\delta>0$ there are trigonometric polynomials $(p_1(t),p_2(t))$ of degree $N$, all of whose coefficients have absolute value less than $\delta$, such that $(x(t)+p_1(t),y(t)+p_2(t))$ is 1-covered and has only finitely many self-intersections.
\end{lemma}
\begin{proof}
Since $f$ is real-analytic and non-constant, at least one of the coordinate functions, say $x(t)$, has only finitely many critical points, say $N'$ many. Via trigonometric interpolation we may then find a trigonometric polynomial $p$ of degree $N:=\left\lfloor\tfrac{N'}{2}\right\rfloor$ such that for all $a>0$ we have $y'(t)+a p'(t)\neq 0$ for all $t$ with $x'(t)=0$. In particular, $F(t):=(x(t),y(t)+a p(t))$ is a regular curve for all  $a>0$. The interpolation condition is simply that $p'(t)=0$ if $x'(t)=0$, but $y'(t)\neq 0$, and $p'(t)$ is non-zero if $x(t)=0$ and $y'(t)=0$.

Lemma~\ref{Nself} implies that the image of $F$ has only finitely many self-intersections. From the proof of Lemma~\ref{Nself} we know that if $F$ is not 1-covered, then it is $N$-covered for some $N>1$. If $F$ is 1-covered, we are done. Assume that $F$ is $N$-covered for some $N>1$. Since the image of $F$ is a regular curve, forming a closed loop, there must be some $t_*\in[0,2\pi]$, where its curvature $\kappa(t)=\tfrac{|(y+ap)'(t)x''(t)-(y+ap)''(t)x'(t)|}{\sqrt{x'(t)^2+(y+ap)'(t)^2}^3}$ is non-zero.

Since $\kappa(t_*)$ is non-zero, there is a neighbourhood $U$ of $t_*$ and a neighbourhood $V$ of $F(t_*)$ in $\mathbb{R}^2$ such that the tangent line $\ell$ through $F(t_*)$ divides $V$ in two components and $F(U\backslash\{t_*\})$ is contained in one of them. Let $\mathrm{N}(t_*)$ be the normal vector of the curve $F(t)$ at $t=t_*$. Now consider the tuple of trigonometric polynomials $-\cos(t-t_*)\mathrm{N}(t_*)$. 

We claim that $F-b\cos(t-t_*)\mathrm{N}(t_*)$ has the property that $F(t_*)-b\mathrm{N}(t_*)$ has a unique preimage point if $b>0$ is sufficiently small. Note that adding $-b\cos(t-t_*)\mathrm{N}(t_*)$ corresponds to a parallel translation of the curve, where the length of the translation is determined by $\cos(t-t_*)$. Since $\cos(t-t_*)$ attains its unique maximum at $t=t_*$, no other point in $F^{-1}(V)$ can map to $F(t_*)-b\mathrm{N}(t_*)$. Note that $F^{-1}(V)$ does not only consist of $U$, but also of $N-1$ other intervals with the same image. By choosing $b$ small enough and potentially shrinking $V$, we can guarantee that $F(t_*)-b\mathrm{N}(t_*)\in V$ and $F([0,2\pi]\backslash F^{-1}(V))\cap V=\emptyset$. Thus $F-b\cos(t-t_*)\mathrm{N}(t_*)$ is 1-covered and has only finitely many self-intersections.

All added terms are trigonometric polynomials of degree at most $N$ and by choosing $a$ and $b$ sufficiently small, we can achieve the bound on the coefficients.
\end{proof}

Note that Lemma~\ref{curvature} goes beyond typical density arguments of trigonometric polynomials. In particular, the degree $N$ only depends on the curve $f$ and is independent of $\delta$.

\begin{lemma}\label{lem:zcoord}
Let $f(t)=(x(t),y(t),z(t)):S^1\to\mathbb{R}^3$ be a smooth map, such that $(x(t),y(t)):S^1\to\mathbb{R}^2$ is 1-covered and only has finitely many self-intersections. Then there is a trigonometric polynomial $p:S^1\to\mathbb{R}$ such that $(x(t),y(t),z(t)+cp(t))$ parametrises a knot for all sufficiently small $c>0$.
\end{lemma}
\begin{proof}
If $f$ parametrises a knot, we can simply take $p=0$.
Suppose that $f$ does not parametrise a knot. Since $(x(t),y(t))$ is 1-covered, all self-intersections of $f$ are self-intersections of its image and not the result of some pre-composed covering map. Since $(x(t),y(t))$ only has finitely many self-intersections, so does $f(t)$. Note that we do not know if $(x(t),y(t))$ parametrises a regular knot diagram. The ``crossings'' could be tangential or could involve more than two strands.

We find the desired trigonometric polynomial $p$ via trigonometric interpolation. Let $c_i=(c_{i,1},c_{i,2})$, $i=1,2,\ldots,m$, denote the self-intersections of $(x(t),y(t))$ and let $t_{i,j}$, $j=1,2,\ldots,m_i$, be the values of $t$ in $[0,2\pi)$ so that $(x(t_{i,j}),y(t_{i,j}))=c_i$ for all $i\in\{1,2,\ldots,m\}$, $j\in\{1,2,\ldots,m_i\}$. If all $z(t_{i,j})$s with a fixed $i$ are distinct we can set $p(t_{i,j})=0$ for all $j\in\{1,2,\ldots,m_i\}$. If there is some $i$ and a subset $J\subset\{1,2,\ldots,m_i\}$ with $|J|>1$ so that $z(t_{i,j})=z(t_{i,j'})$ for all $j,j'\in J$, then we demand that $p(t_{i,j})\neq p(t_{i,j'})$ for all $j,j'\in J$. Since the set of interpolation points is finite, such a trigonometric polynomial $p$ always exists and by construction we have $z(t_{i,j})+cp(t_{i,j})\neq z(t_{u,j'})+cp(t_{i,j'})$ for all $i\in\{1,2,\ldots,m\}$, $j,j'\in\{1,2,\ldots,m_i\}$, as long as $c$ is sufficiently small. By construction, the resulting curve has no self-intersections and thus is a knot.
\end{proof}

\begin{lemma}
Let $f:S^1\to\mathbb{R}^3$ be as in Proposition~\ref{prop}. Then there is a smooth function $g:S^1\to\mathbb{R}^3$, whose coordinate functions are smooth infinite Fourier series, such that $f+g$ parametrises $K$, its loop sequence contains only knots and $(f+g)^{(n_i)}$ parametrises $(L_i)_{i\in\mathbb{N}}$ for all $i\in\mathbb{N}$.
\end{lemma}
\begin{proof}
Suppose that $f^{(1)}$ does not parametrise a knot. Since $f^{(n_i)}$ parametrises a knot for every $i$, there are at least two coordinate functions, say $(x'(t),y'(t))$, of $f^{(1)}$ that are not constant. By Lemma~\ref{curvature} there are trigonometric polynomials $p_{1,1}$ and $p_{1,2}$ such that $(x'(t)+p_{1,1}(t),y'(t)+p_{1,2}(t))$ is 1-covered and has only finitely many self-intersections. Then by Lemma~\ref{lem:zcoord} there is a trigonometric polynomial $p_{1,3}$ such that $(x'(t)+p_{1,1}(t),y'(t)+p_{1,2}(t),z'(t)+p_{1,3}(t))$ parametrises a knot.

We can write the trigonometric polynomials are written as finite Fourier series as $p_{1,j}(t)=\sum_{k=-\infty}^{\infty}a_{1,j,k}\rme^{\rmi kt}$, with all but finitely many $a_{1,j,k}$ equal to zero. Note that we can assume that $a_{1,j,0}=0$ for all $j=1,2,3$, since changing the constant term corresponds to an overall translation of the parametrised knot. Then we define
\begin{equation}
\widetilde{p}_{1,j}(t):=\sum_{k=-\infty}^\infty \frac{a_{1,j,k}}{\rmi k}\rme^{\rmi kt}
\end{equation}
and $\widetilde{p}_1(t)=(\widetilde{p}_{1,1}(t),\widetilde{p}_{1,2}(t),\widetilde{p}_{1,3}(t))$. It follows that 
\begin{equation}
(f(t)+\widetilde{p}_1(t))'=f'(t)+(p_{1,1}(t),p_{1,2}(t),p_{1,3}(t))
\end{equation} 
and in particular, it parametrises a knot, say $K_1$.

We show that $\widetilde{p}_1$ can be chosen such that $f+\widetilde{p_1}$ still parametrises $K$ and $(f+\widetilde{p})^{(n_i)}$ still parametrises $L_i$ for all $i$. We know that adding Fourier series to parametrisations of knots does not change knot types if the added coefficients are sufficiently small. The construction from the previous section was based on trigonometric parametrisations of $K$ and each $L_i$. Recall that we associate to every such parametrisation an ``$\varepsilon$-value'' $\varepsilon_i$ (with $\varepsilon_0$ for $K$), so that it is suffices to prove that an added coefficient of $\rme^{\rmi kt}$ is (in absolute value) less than $\varepsilon_i/2^{|k|}$. By rescaling the knots, we can assume that all these $\varepsilon$-values are equal to 1. 

The construction of $f$ in the previous section is done in such a way that the $k$th coefficients differ from the original parametrisations of $K$ by less than $1/2^{|k|+1}$, see Remark~\ref{rem2}. Thus we have to show that the coefficient of $\rme^{\rmi kt}$ that we add now has an absolute value less than $1/2^{|k|+1}$. Then the triangle inequality implies that we have not changed the knot $K$. Likewise, the coefficients of $f^{(n_i)}$ differ from those in $(N_i-1)^{n_i}f_i$, the parametrisation of $L_i$, by at most $(N_i-1)^{n_i}/2^{|k|+1}$. In order to preserve the knot $L_i$ we thus need that
\begin{equation}\label{eq:ineqa1jk}
|a_{1,j,k}||k|^{n_i-1}<(N_i-1)^{n_i}/2^{|k|+1}
\end{equation}
for all $j=1,2,3$, $k\in\mathbb{Z}$. 

By Lemma~\ref{curvature} and Lemma~\ref{lem:zcoord} we can choose the absolute values of the coefficients $a_{1,j,k}$ arbitrarily small without changing the degree $N$ of the trigonometric polynomials, so that $|a_{1,j,k}|=0$ for all $k$ with $|k|>N$. Since the sequence $(N_i-1)^{n_i}|k|^{1-n_i}$ goes to infinity as $i$ goes to infinity (keeping $k$ fixed), knowing that $N_i\geq N_1>1$, the sequence has a positive global minimum. Thus by choosing the coefficients $a_{1,j,k}$ sufficiently small, we can guarantee that Eq.~\eqref{eq:ineqa1jk} is satisfied for all $j,k$ and all $i$ simultaneously.

In fact, we may choose $|a_{1,j,k}||k|^{n_i-1}<(N_i-1)^{n_i}/2^{|k|+2}$ for all $i$ and $|a_{1,j,k}/k|<1/2^{|k|+2}$. Thus $f+\widetilde{p}_1$ parametrises $K$ and $(f+\widetilde{p}_1)^{(n_i)}$ parametrises $L_i$ for all $i$. Note that in this step we have used  that we can choose the coefficients of the trigonometric polynomials $p_{1,j}$, $j=1,2,3$, small without changing their degrees.

If $f^{(1)}(t)$ already parametrises a knot, we simply set $\widetilde{p}_1(t)=(0,0,0)$.

We write $\widetilde{\varepsilon}_1$ for the $\varepsilon$-value of the knot $K_1$, parametrised by $(f(t)+\widetilde{p}_1(t))'$.

We proceed inductively. Suppose that we have found a finite set of triples of trigonometric polynomials $\widetilde{p}_{j}(t):=(\widetilde{p}_{j,1}(t),\widetilde{p}_{j,2}(t),\widetilde{p}_{j,3}(t))$ for all $j=1,2,\ldots,n$ for some natural number $n$ such that $(f(t)+\sum_{j=1}^i\widetilde{p}_j(t))^{(i)}$ parametrises a knot $K_i$ for all $i=1,2,\ldots,n$. We write $\widetilde{\varepsilon}_i$ for the $\varepsilon$-value of $K_i$. Furthermore, we assume that the coefficients of $\rme^{\rmi kt}$ in $f(t)+\sum_{j=1}^n\widetilde{p}_j(t)$ differ from those in $f$ by less than
\begin{equation}
\min\left\{\sum_{j=1}^n\frac{1}{2^{|k|+j+1}},\min_{i\in\{1,2,\ldots,n-1\}}\sum_{j=i+1}^{n}\frac{\widetilde{\varepsilon}_i}{2^{|k|+j-i+1}}\right\}.
\end{equation}
This implies that $f(t)+\sum_{j=1}^n\widetilde{p}_j(t)$ parametrises $K$ and $(f(t)+\sum_{j=1}^n\widetilde{p}_j(t))^{(n_i)}$ parametrises $L_i$ for all $i$. Furthermore, it shows that $(f(t)+\sum_{j=1}^n\widetilde{p}_j(t))^{(i)}$ still parametrises the same knot $K_i$ as $(f(t)+\sum_{j=1}^i\widetilde{p}_j(t))^{(i)}$.

If $(f(t)+\sum_{j=1}^n\widetilde{p}_j(t))^{(n+1)}$ already parametrises a knot, we set $\widetilde{p}_{n+1}(t)=(0,0,0)$. Otherwise, we find by the same arguments as above a triple of trigonometric polynomials $p_{n+1}(t):=(p_{n+1,1}(t),p_{n+1,2}(t),p_{n+1,3}(t))$ such that $(f(t)+\sum_{j=1}^{n}\widetilde{p}_{j}(t))^{(n+1)}+p_{n+1}(t)$ parametrises a knot.

We write $p_{n+1,j}(t)$, $j=1,2,3$, as a finite Fourier series $\sum_{k=-\infty}^{\infty}a_{n+1,j,k}\rme^{\rmi kt}$ with all but finitely many coefficients equal to zero and $a_{n+1,j,0}=0$, and define
\begin{equation}
\widetilde{p}_{n+1,j}(t)=\sum_{k=-\infty}^{\infty}\frac{a_{n+1,j,k}}{(\rmi k)^{n+1}}\rme^{\rmi kt}
\end{equation}
as well as $\widetilde{p}_{n+1}(t):=(\widetilde{p}_{n+1,1}(t),\widetilde{p}_{n+1,2}(t),\widetilde{p}_{n+1,3}(t))$. Then $(f(t)+\sum_{j=1}^{n+1}\widetilde{p}_{j}(t))^{(n+1)}$ parametrises the same knot $K_{n+1}$ as $(f(t)+\sum_{j=1}^{n}\widetilde{p}_{j}(t))^{(n+1)}+p_{n+1}(t)$.

Furthermore, we may choose the absolute values of the coefficients of $\rme^{\rmi kt}$ in $\widetilde{p}_{n+1}(t)$, so that
\begin{align}
|a_{n+1,j,k}||k|^{-n-1}&<2^{-|k|-n-2},&\\
|a_{n+1,j,k}||k|^{n_i-n-1}&<(N_i-1)^{n_i}2^{-|k|-n-2}&\text{ for all }i,\label{eq:2ineq}\\
|a_{n+1,j,k}||k|^{i-n-1}&<\widetilde{\varepsilon}_i2^{-|k|-n-2+i} &\text{ for all }i\in\{1,2,\ldots,n\}.
\end{align}
The fact that we can satisfy Eq.~\eqref{eq:2ineq} for all $i$ simultaneously is proved by the same arguments as in the initial step above for $\widetilde{p}_1$. The inequalities imply that $(f(t)+\sum_{j=1}^{n+1}\widetilde{p}_{j}(t))$ parametrises $K$, $(f(t)+\sum_{j=1}^{n+1}\widetilde{p}_{j}(t))^{(n_i)}$ parametrises $L_i$ for all $i$, and $(f(t)+\sum_{j=1}^{n+1}\widetilde{p}_{j}(t))^{(i)}$ parametrises a knot $K_i$ for all $i=1,2,\ldots,n$.

Thus we have an inductive definition of trigonometric polynomials $(\widetilde{p}_j)_{j\in\mathbb{N}}$ such that the coefficients of $\rme^{\rmi kt}$ in $(f(t)+\sum_{j=1}^{\infty}\widetilde{p}_j(t))^{(n_i)}$ differ from the corresponding coefficients in $(N_i-1)^{n_i}f_i(t)$, which was the original parametrisation of $L_i$, by less than
\begin{equation}
(N_i-1)^{n_i}\sum_{j=0}^{\infty}1/2^{|k|+j+1}=(N_i-1)^{n_i}1/2^{|k|},
\end{equation}
where the $j=0$-term comes from the construction of $f$ in the previous section, while the terms with $j>0$ come from $\widetilde{p}_j$. In particular, by Lemma~\ref{lem:epsilon} $(f(t)+\sum_{j=1}^{\infty}\widetilde{p}_j(t))^{(n_i)}$ is a parametrisation of $L_i$. Likewise, setting $n_0=0$, we get that $f(t)+\sum_{j=1}^{\infty}\widetilde{p}_j(t)$ is a parametrisation of $K$. For all $n$ with $n\neq n_i$ for all $i\in\mathbb{N}$ we have that the coefficients of $\rme^{\rmi kt}$ in $(f(t)+\sum_{j=1}^{\infty}\widetilde{p}_j(t))^{(n)}$ differ from the corresponding coefficients in $(f(t)+\sum_{j=1}^{n}\widetilde{p}_j(t))^{(n)}$ by less than
\begin{equation}
\widetilde{\varepsilon}_n\sum_{j=n+1}^{\infty}1/2^{|k|+j}<\widetilde{\varepsilon}_n/2^{|k|}.
\end{equation}
So in particular, again by Lemma~\ref{lem:epsilon}, $(f(t)+\sum_{j=1}^{\infty}\widetilde{p}_j(t))^{(n)}$ parametrises a knot, namely $K_n$, the same knot that is parametrised by  $(f(t)+\sum_{j=1}^{n}\widetilde{p}_j(t))^{(n)}$. Since this is true for all $n\in\mathbb{N}$, this finishes the proof.
\end{proof}

This concludes the proof of Theorem~\ref{thm:ultra}. As explained in the introduction an immediate consequence of this result is that every knot is an ultraknot.

Having shown that we can prescribe subsequences $K_{n_i}$ of a  loop sequence of a parametrisation $f$ of any given knot $K$, it is a natural question whether we can realise any sequence of knots $K_n$ as the loop sequence of a parametrisation $f$ of any given knot $K$. This is still an open problem. Note that the methods from the previous section are not well-suited for this problem, since the argument relies on the fact that we can make the gaps between the different $n_i$s arbitrarily large.

\section{Limit knots}\label{sec:4}

In this section we study limit knots, the knot types that can arise as limits of sequences of knots $K_n$, parametrised by $f^{(n)}$ for some smooth function $f:S^1\to\mathbb{R}^3$. In general, the question which knots arise as limit knots remains a difficult problem. However, if we assume that the coordinate functions of $f$ are trigonometric polynomials, so that $f$ is the parametrisation of a Fourier knot, we obtain several results.

We start our discussion of limit knots with the example of the unknot. It can be parametrised as the planar unit circle $f(t)=(\cos(t),\sin(t),0)$. Then the sequence of parametric curves $K_n$, parametrised by $f^{(n)}(t)$, is periodic with period 4, i.e., $K_n=K_{n+4}$ for all $n\in\mathbb{N}$. We have $f^{(4k)}(t)=f(t)$, as well as
\begin{align}
f^{(4k+1)}(t)&=(-\sin(t),\cos(t),0),\nonumber\\
f^{(4k+2)}(t)&=(-\cos(t),-\sin(t),0)=-f(t),\nonumber\\
f^{(4k+3)}(t)&=(\sin(t),-\cos(t),0)=-f^{(4k+1)}(t).
\end{align}
Note that all of these loops are again the unit circle in the plane. The only thing that changes is the $t=0$-point on the curve, which is rotated by $\pi/2$ with each derivative. Thus the unknot is a limit knot.

It follows from Theorem~\ref{thm:ultra} that every knot is the limit of a convergent subsequence  of $K_n$, parametrised by $f^{(n)}$ for some initial knot parametrisation $f$, since we can pick $K_{n_i}=L_i=K$ for all $i$. However, since we are only dealing with sequence of knots that are realised by loop sequences of smooth maps $f:S^1\to\mathbb{R}^3$ and these subsequences are in general not of this form, this does not imply that every knot is a limit knot.

\begin{lemma}\label{lem:liss}
Let $f:S^1\to\mathbb{R}^3$, $f(t)=(x(t),y(t),z(t))$, where each coordinate function is a real trigonometric polynomial whose highest order terms are given by 
\begin{align}
&A_{x,N_x}\cos(N_xt),\nonumber\\
&A_{y,N_y}\cos(N_yt+\varphi_{y,N_y},\\
&A_{z,N_z}\cos(N_zt+\varphi_{z,N_z}),\nonumber
\end{align}
respectively, with $A_{x,N_x}, A_{y,N_y},A_{z,N_z}\in\mathbb{R}\backslash\{0\}$, $n_x,n_y,n_z\in\mathbb{N}$ and $\varphi_y,\varphi_z\in[0,2\pi)$. Assume that the frequencies $n_x$, $n_y$ and $n_z$ are pairwise coprime. Furthermore, assume that $\varphi_y/\pi,\varphi_z/\pi,(\varphi_y-\varphi_z)/\pi\notin\mathbb{Q}$. Then there is an $M>0$ such that for all $m>M$ the knot $K_m$ is a Lissajous knot. 
\end{lemma}
\begin{proof}
Every real trigonometric polynomial $p$ can be written in the form $p(t)=\sum_{k=0}^NA_k\cos(kt+\varphi_k)$, where $N\in\mathbb{N}\cup\{0\}$ is the degree of the trigonometric polynomial, $A_k\in\mathbb{R}$ and $\varphi_k\in[0,2\pi)$ for all $k$. If $p$ is non-constant, then the degree of the $n$th derivative of $p$ is again $N$ for all $n\in\mathbb{N}$. The coefficient of $p^{(n)}$ corresponding to the frequency $k$ has absolute value $k^n|c_k|$, so that eventually (for sufficiently large $n$) the coefficient corresponding to $k=N$ is much larger than all other coefficients.

If the highest order terms of $f$ are as stated above, then the highest order terms of $f^{(4\ell+i)}$, $i=0,1,2,3$, are
\begin{align}
&X_{4\ell+i}(t)=N_x^{4\ell+i}A_{x,N_x}\cos(N_xt+\varphi_{x,N_x}+i\pi/2),\nonumber\\ 
&Y_{4\ell+i}(t)=N_y^{4\ell+i}A_{y,N_y}\cos(N_yt+\varphi_{y,N_y}+i\pi/2),\\
&Z_{4\ell+i}(t)=N_z^{4\ell+i}A_{z,N_z}\cos(N_zt+\varphi_{z,N_z}+i\pi/2).\nonumber
\end{align}

The article \cite{272} completely characterises the Lissajous curves with self-intersections. The assumption on the highest order terms implies that
\begin{equation}
(N_x^{-(4\ell+i)}X_{4\ell+i}(t),N_y^{-(4\ell+i)}Y_{4\ell+i}(t),N_z^{-(4\ell+i)}Z_{4\ell+i}(t))
\end{equation} 
parametrises a knot (as opposed to a non-simple loop) $K'_{i}$, which by definition is a Lissajous knot. We claim that for large enough $\ell$ the knot $K_{4\ell+i}$, which is parametrised by $f^{(4\ell+i)}(t)$, is equivalent to $K'_{i}$.

By Lemma~\ref{lem:epsilon} there is an $\varepsilon>0$ such that adding Fourier series $\sum_{k\in\mathbb{Z}}a_{x,k}\rme^{\rmi kt}$ to the coordinate functions does not change the knot type of $K'_i$ as long as the absolute values of the coefficients $a_{x,k}$ are bounded from above by $\varepsilon/2^{|k|}$ (and similarly for $y$ and $z$).

We obtain a new parametrisation of $K_{4\ell+i}$ by multiplying the $x$-coordinate by $N_x^{-(4\ell+i)}$, the $y$-coordinate by $N_y^{-(4\ell+i)}$ and the $z$-coordinate by $N_z^{-(4\ell+i)}$. Thus $K_{4\ell+i}$ has a trigonometric parametrisation, where the highest order terms are precisely $(N_x^{-(4\ell+i)}X_{4\ell+i}(t),N_y^{-(4\ell+i)}Y_{4\ell+i}(t),N_z^{-(4\ell+i)}Z_{4\ell+i}(t))$. Writing $x(t)=\sum_{k=0}^{N_x}A_{x,k}\cos(kt+\varphi_{k})$, the lower order terms  of this new parametrisation are of the form $\left(\tfrac{k}{N_x}\right)^{4\ell+i}A_{x,k}$ (and similarly for $y$ and $z$). If $\ell$ is sufficiently large, this has a smaller absolute value than $\varepsilon/2^{|k|}$, because $N_x>k$. Since there are only finitely many non-zero $A_{x,k}$s, such a value of $\ell$ exists for all $k$ with $|k|\leq \max\{N_x,N_y,N_z\}$ at the same time. Thus by Lemma~\ref{lem:epsilon} the knot $K_{4\ell+i}$ is a sufficiently close approximation of $K'_{i}$ to be ambient isotopic.
\end{proof}

Lemma~\ref{lem:liss} proves the first part of Theorem~\ref{thm:limit}. Since after some point all knots in the sequence are Lissajous knots, the limit knot (if it exists) must also be a Lissajous knot. The assumption that the maximal degrees $N_x$, $N_y$ and $N_z$ are pairwise coprime is necessary to guarantee that $K'_i$ is actually a knot and not simply a loop with intersections or a loop that traces back on itself. The conditions on the phase shifts can be somewhat relaxed, see \cite{272} for a complete description of the values of $\varphi_y$ and $\varphi_z$ that result in self-intersections.

At the moment it is not clear, which knots can be obtained as limit knots from triples of trigonometric polynomials whose maximal degrees are not pairwise coprime or whose phase shifts are not as in the lemma. It is conceivable that $K'_i$ is a singular knot, whose singular crossings are resolved by the lower order terms and in principle one could obtain different knot types for different derivatives. 

Lissajous knots have specific symmetries, see \cite{272}. If all frequencies of a Lissajous knot $K$ are odd, then $K$ must be strongly plus amphicheiral and if one of the frequencies is even, then $K$ must be 2-periodic. Since the frequencies are pairwise coprime, these are the only two possible cases. Furthermore, the Arf invariant of a Lissajous knot must be zero. This implies that there are knots, such as the trefoil knot or the figure-eight knot, that are not Lissajous knots and hence are not Fourier limit knots. However, there are infinitely many Lissajous knot types \cite{infty} and it was proved that every knot can be parametrised by a trigonometric function that has only one term in its $x$-coordinate, one terms in its $y$-coordinate and two terms in its $z$-coordinate \cite{soret}.


\begin{lemma}\label{lem:limit}
Let $K$ be a Lissajous knot such that the frequencies $n_x$, $n_y$ and $n_z$ in its parametrisation are all odd. Then $K$ is a Fourier limit knot.
\end{lemma}
\begin{proof}
We write $K_0=K$ and write $K_i$, $i=1,2,3$, for the knots that are parametrised by $f^{(i)}$, the $i$th derivative of the trigonometric polynomial defining $K$ as a Lissajous knot. For a more compact notation we write $x_1$ for $x$, $x_2$ for $y$ and $x_3$ for $z$. Likewise, the frequencies and phase shifts are (for example) denoted by $n_1$ and $\varphi_1$ instead of $n_x$ and $\varphi_x$, respectively. In particular, $K_0$ is given by 
\begin{align}\label{eq:lisspara}
x_1(t)&=\cos(n_1t+\varphi_1),\nonumber\\
x_2(t)&=\cos(n_2t+\varphi_2),\nonumber\\
x_3(t)&=\cos(n_3t+\varphi_3).
\end{align}
Then $K_1$ is ambient isotopic to 
\begin{equation}
(\cos(n_1t+\varphi_1+\pi/2),\cos(n_2t+\varphi_2+\pi/2),\cos(n_3t+\varphi_3+\pi/2)).
\end{equation} 
In fact, every $K_i$ is ambient isotopic to a shift in all cosines by $\pi i/2$.

It is easy to see that
\begin{equation}
\cos(n_jt+\varphi_j+i\pi/2)=\cos(n_jt+\varphi_j+(4k_j+i)\pi/2)
\end{equation} 
for every $k_j\in\mathbb{Z}$ and all $i,j$. We may pick $k_j=\left\lfloor\tfrac{n_j}{4}\right\rfloor$, where $\lfloor\cdot\rfloor$ is the floor function that maps every real number $w$ to the largest integer that is less than or equal to $w$. It follows that $K_1$ is ambient isotopic to
\begin{equation}\label{eq:deltapara}
(\cos(n_1(t+\pi/2)+\varphi_1+\delta_1),\cos(n_2(t+\pi/2)+\varphi_2+\delta_2),\cos(n_3(t+\pi/2)+\varphi_3+\delta_3)),
\end{equation}
where $\delta_j=0$ if $n_j\equiv 1\text{ mod }4$ and $\delta_j=\pi$ if $n_j\equiv 3\text{ mod }4$. Therefore, every coordinate function in Eq.~\eqref{eq:deltapara} differs from the corresponding function in Eq.~\eqref{eq:lisspara} by a shift of $\pi/2$ in the variable $t$ and possibly an overall sign, depending on the residue class of $n_j$ modulo 4. Since Eq.~\eqref{eq:lisspara} parametrises $K$, this shows that $K_1$ is ambient isotopic to $K$ (if an even number of frequencies $n_j$ are 3 mod 4) or ambient isotopic to the mirror image of $K$ (if an odd number of frequencies $n_j$ are 3 mod 4). Since $K$ is a Lissajous knot with only odd frequencies, it is equivalent to its mirror image, so that in any case $K_1$ is ambient isotopic to $K$.
 
Note that $K_2$ is the mirror image of $K$ and therefore ambient isotopic to $K$. Likewise, $K_3$ is the mirror image of $K_1$ and therefore also ambient isotopic to $K$.

Since the only difference between $K_{i}$ and $K_{4k+i}$, $k\in\mathbb{N}$ is an overall linear factor that does not affect the knot type, we have $K_i\cong K$ for all $i$ and hence $K$ is a limit knot.

We may set $\varphi_1=0$ and vary $\varphi_2$ and $\varphi_3$ slightly without changing the knot type $K$. In particular, we can run through the same arguments as above for a Lissajous parametrisation of $K$, where $\varphi_2/\pi, \varphi_3/\pi,(\varphi_2-\varphi_3)/\pi\notin\mathbb{Q}$. It follows that $K$ is a Fourier limit knot.
\end{proof}
Together with Lemma~\ref{lem:liss} this completes the proof of Theorem~\ref{thm:limit}.

Note that the assumption that all frequencies are odd is necessary for the argument to work. Consider for example $f(t)=(x(t),y(t),z(t))$ with
\begin{align}
x(t)&=\cos(2t),\nonumber\\
y(t)&=\cos(3t+0.56099),\\
z(t)&=\cos(11t+2.58059),\nonumber
\end{align}
which parametrises the knot $5_2$ \cite{sample}. (Note that \cite{sample} uses the Hoste-Thistlethwaite-Weeks table \cite{hoste}, while we describe knots by their label in Rolfsen's table \cite{rolfsen}.) The derivative $f'(t)$ parametrises the unknot. The corresponding curves are shown in Figure~\ref{fig}. Therefore, this particular Lissajous parametrisation does not induce a constant loop sequence. Still, it might be possible to find a different (Lissajous) parametrisation that establishes $5_2$ as a (Fourier) limit knot.

\begin{figure}
\centering
\includegraphics[height=4cm]{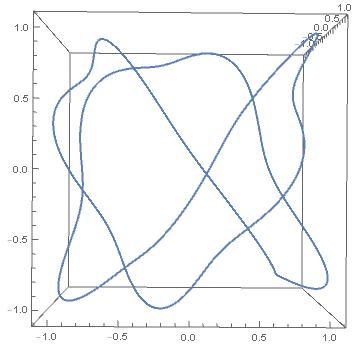}\qquad
\includegraphics[height=4cm]{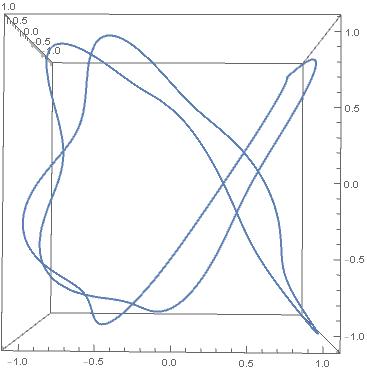}\\
\labellist
\Large
\pinlabel a) at 100 1850
\pinlabel b) at 1200 1850
\pinlabel c) at 100 850
\pinlabel d) at 1200 850
\endlabellist
\includegraphics[height=4cm]{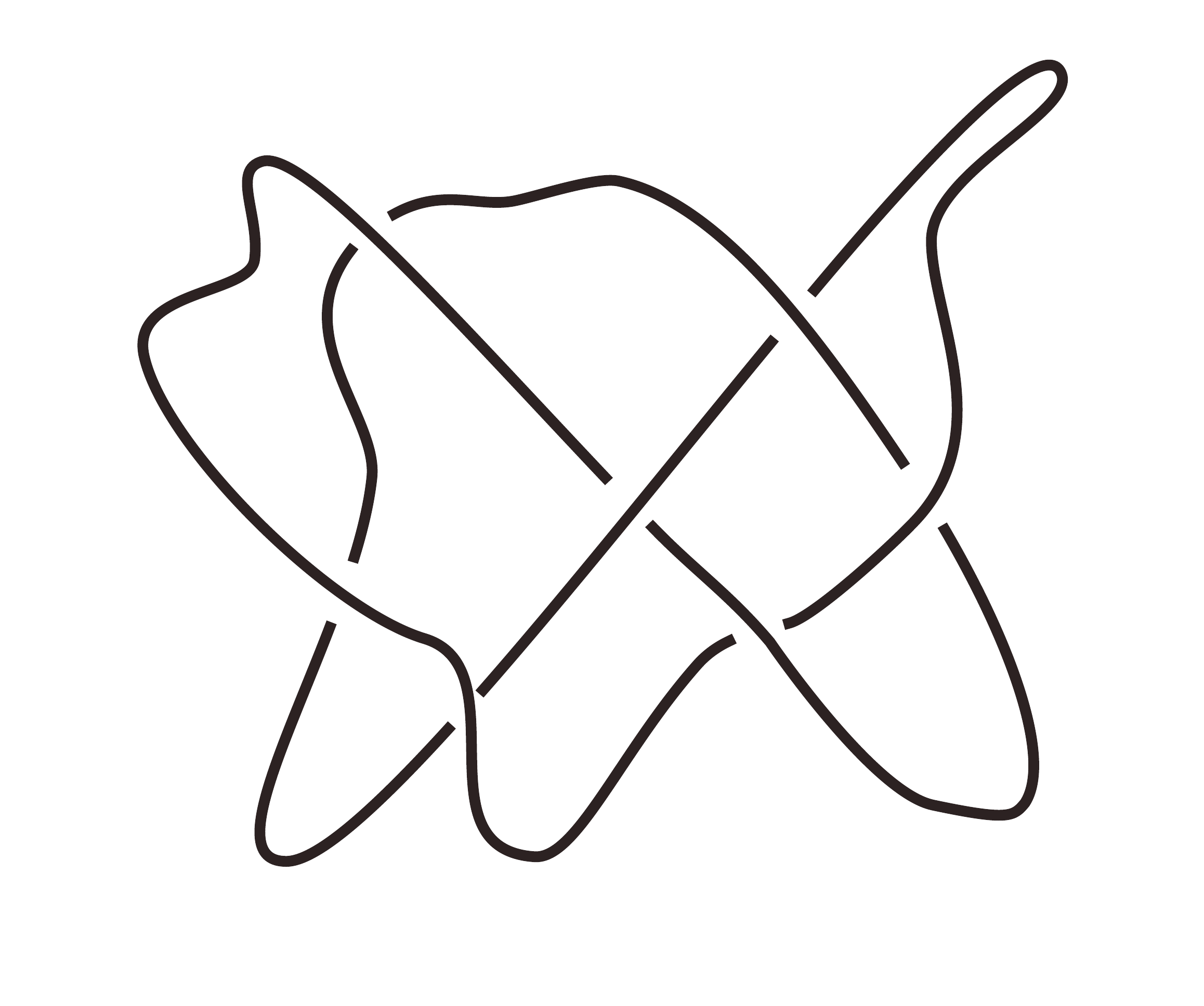}
\includegraphics[height=4cm]{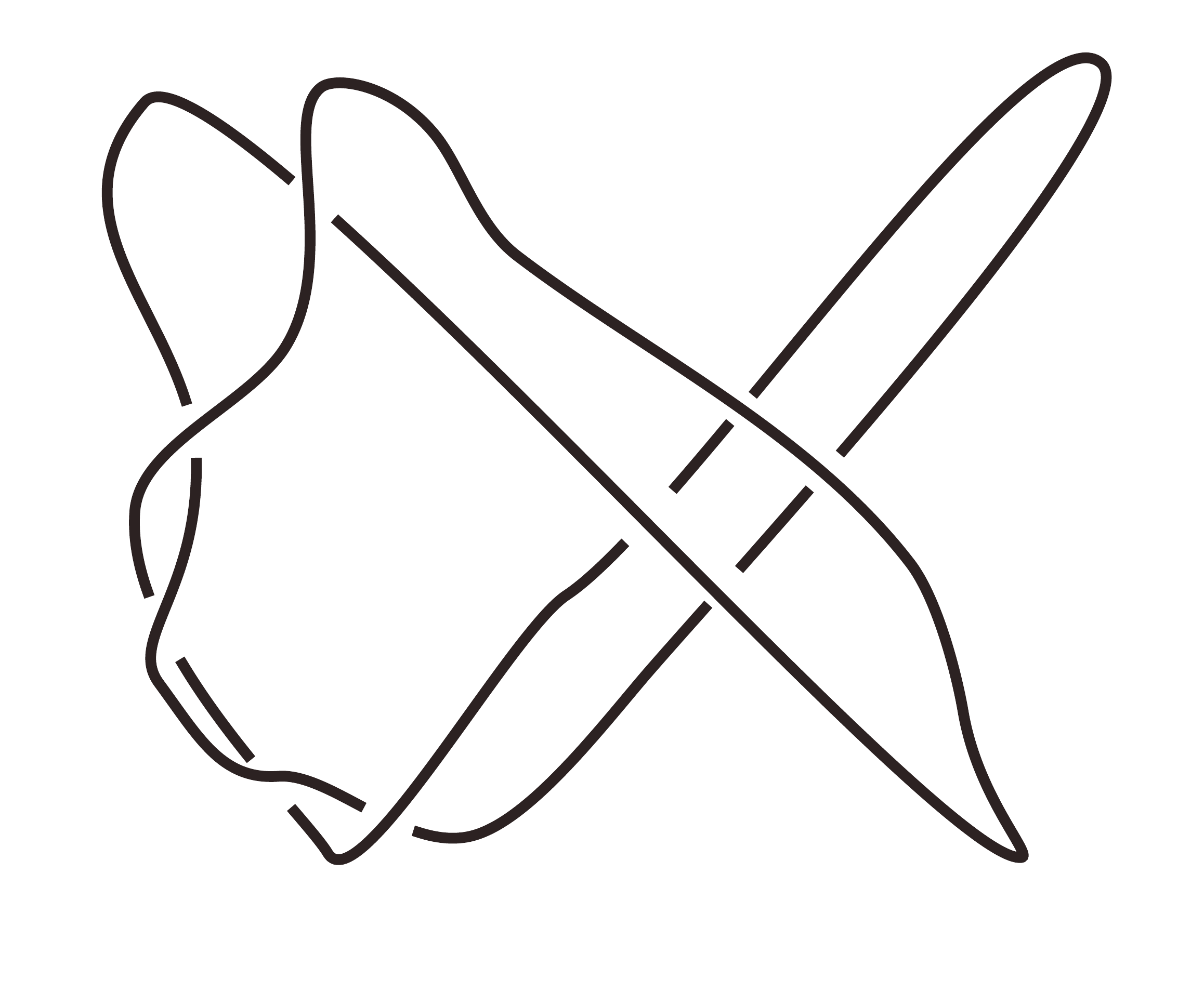}
\caption{a) The knot $5_2$ parametrised by $f(t)$. b) The derivative $f'(t)$ parametrises the unknot. c) The knot diagram for $5_2$. d) The knot diagram for the unknot. \label{fig}}
\end{figure}

By definition limit knots can be obtained as limits of constant sequences, that is, every derivative of the initial function yields the same knot. We now study the following question: Given a Lissajous Fourier limit knot $K$ what are the knots $K_0$ such that there is a parametrisation of $K_0$ whose corresponding sequence of knots $(K_n)_{n\in\mathbb{N}}$ converges to $K$?

\begin{proposition}\label{prop:start}
Let $K$ be a Lissajous knot with only odd frequencies $n_x$, $n_y$ and $n_z$. Let $B$ be braid on $s$ strands with $\ell$ crossings such that $s+\left\lfloor\tfrac{s\ell}{2}\right\rfloor<\min\{n_x,n_y\}$ and $(s+1) s\ell/2<n_z$ and such that the closure of $B$ is a knot $K_0$. Then there is a parametrisation $f:S^1\to\mathbb{R}^3$ of $K_0$ such that the sequence of knots $(K_n)_{n\in\mathbb{N}}$, parametrised by $f^{(n)}$, converges to $K$, i.e., there is an $N\in\mathbb{N}$ such that for all $n>N$ we have $K_n\cong K$. 
\end{proposition}
\begin{proof}
It was shown in \cite{bode:algo} that a braid that is isotopic to $B$ can be parametrised as
\begin{equation}
\bigcup_{t\in[0,2\pi]}\bigcup_{j=1}^s\left(F\left(\frac{t+2\pi j}{s}\right),G\left(\frac{t+2\pi j}{s}\right),t\right)\subset\mathbb{R}^2\times[0,2\pi],
\end{equation}
where $F$ and $G$ are trigonometric polynomials of degree at most $\left\lfloor\tfrac{s\ell}{2}\right\rfloor$ and $(s+1)s\ell/2$, respectively. In \cite{bode:algo} the bounds are given erroneously as $\left\lfloor\tfrac{s\ell-1}{2}\right\rfloor$ and $\left\lfloor\tfrac{s+1}{2}\right\rfloor (s\ell-1)$, respectively. This mistake was pointed out in \cite{bode:ak}.

From this braid parametrisation we obtain a Fourier parametrisation of its closure $K_0$ via
\begin{align}
x(t)&=\cos(st)(R+F(t)),\nonumber\\
y(t)&=\sin(st)(R+F(t)),\nonumber\\
z(t)&=G(t),
\end{align}
where $R$ is some large positive real number. The maximum of the degrees of $x$ and $y$ is thus at most $s+\left\lfloor\tfrac{s\ell}{2}\right\rfloor$, while the degree of $z$ is at most $(s+1)s\ell/2$.

Let $(X(t),Y(t),Z(t))$ be the parametrisation of $K$ as a Lissajous knot. By Lemma~\ref{lem:epsilon} there is an $\varepsilon>0$ such that $(x(t)+\varepsilon X(t),y(t)+\varepsilon Y(t),z(t)+\varepsilon Z(t))$ is still a parametrisation of $K_0$. Since the frequencies $n_x$, $n_y$ and $n_z$ are strictly larger than the degrees of $x$, $y$ and $z$, respectively, it follows from the proof of Lemma~\ref{lem:liss} that for all sufficiently large $n$ the knot $K_{4n}$ is equivalent to $K$. Since all frequencies $n_x$, $n_y$ and $n_z$ are odd, the proof of Lemma~\ref{lem:limit} implies that $K_n\cong K$ for all sufficiently large $n$.
\end{proof}

\begin{corollary}\label{cor:unknot}
Let $K$ be a knot. Then there is a parametrisation $f:S^1\to\mathbb{R}^3$ of $K$ such that the resulting sequence of knots $(K_n)_{n\in\mathbb{N}}$ converges to the unknot.
\end{corollary}
\begin{proof}
By \cite{272} the unknot is a Lissajous knot for any triple of pairwise coprime frequencies $(n_x,n_y,n_z)$. The corollary then follows from Proposition~\ref{prop:start}.
\end{proof}



\section*{Acknowledgements}

The author is grateful to Peter Feller for fruitful discussions.
The author is supported by the European Union's Horizon 2020 research and innovation programme through the Marie Sklodowska-Curie grant agreement 101023017.


\begin{thebibliography}{99}



\bibitem{bode:algo} B. Bode and M. R. Dennis. \textit{Constructing a polynomial whose nodal set is any prescribed knot or link}. Journal of Knot Theory and its Ramifications \textbf{28}, no. 1 (2019), 1850082.

\bibitem{bode:ak} B. Bode. \textit{All links are semiholomorphic}. European Journal of Mathematics \textbf{9} (2023), article no. 85.

\bibitem{272} M. G. V. Bogle, J. E. Hearst, V. F. R. Jones and L. Stoilov. \textit{Lissajous knots}. Journal of Knot Theory and its Ramifications \textbf{3}, no. 2 (1994), 121--140.

\bibitem{sample} A. Boocher, J. Daigle, J. Hoste and W. Zheng. \textit{Sampling Lissajous and Fourier knots}. Experimental Mathematics \textbf{18}, no. 4 (2009), 481--497.

\bibitem{hoste} J. Hoste, M. Thistlethwaite and J. Weeks. \textit{The first 1,701,936 knots}. Math. Intelligencer \textbf{20}, no. 4 (1998), 33--48 

\bibitem{kauffman} L. H. Kauffman. \textit{Fourier knots}, In: \textit{Ideal knots}, Series on Knots and Everything 19, eds. A. Stasiak, V. Katritch and L. H. Kauffman (World Scientific, Singapore, 1998), 364--373.

\bibitem{infty} C. Lamm. \textit{There are infinitely many Lissajous knots}. Manuscripta Matematica \textbf{93} (1997), 29--37.

\bibitem{rolfsen} D. Rolfsen. \textit{Knots and links}. Houston: Publish or Perish Inc., 1990. 

\bibitem{soret} M. Soret and M. Ville. \textit{Lissajous and Fourier knots}. Journal of Knot Theory and its Ramifications \textbf{25}, no. 5 (2016), 1650026.



\end{thebibliography}
\end{document}